\documentclass[12pt, reqno]{amsart}
\usepackage{amsmath, amsthm, amscd, amsfonts, amssymb, graphicx, color, xcolor}
\usepackage[bookmarksnumbered, colorlinks, plainpages]{hyperref}
\hypersetup{colorlinks=true,linkcolor=red, anchorcolor=green, citecolor=cyan, urlcolor=red, filecolor=magenta, pdftoolbar=true}

\newtheorem{theorem}{Theorem}[section]

\newtheorem{proposition}[theorem]{Proposition}
\newtheorem{corollary}[theorem]{Corollary}
\theoremstyle{definition}
\newtheorem{definition}[theorem]{Definition}
\newtheorem{example}[theorem]{Example}

\numberwithin{equation}{section}

\begin{document}

\setcounter{page}{1}

\title{$\ast$-K-g-Frames and their duals for Hilbert $\mathcal{A}$-modules}

\author{M'hamed Ghiati$^{1*}$, Samir Kabbaj$^2$, Hatim Labrigui$^2$, Abdeslam Touri$^2$  and Mohamed Rossafi$^3$}

\address{$^{1}$Laboratory Analysis, Geometry and Applications Department of Mathematics, Faculty Of Sciences, University of Ibn Tofail, Kenitra, Morocco}
\email{\textcolor[rgb]{0.00,0.00,0.84}{lalamimouna.mhamed@gmail.com}}

\address{$^{2}$Laboratory of Partial Differential Equations, Spectral Algebra and Geometry Department of Mathematics, Faculty of Sciences, University Ibn Tofail, Kenitra, Morocco}

\email{\textcolor[rgb]{0.00,0.00,0.84}{samkabbaj@yahoo.fr}}

\email{\textcolor[rgb]{0.00,0.00,0.84}{hlabrigui75@gmail.com}}

\email{\textcolor[rgb]{0.00,0.00,0.84}{touri.abdo68@gmail.com}}

\address{$^{3}$LaSMA Laboratory Department of Mathematics Faculty of Sciences, Dhar El Mahraz University Sidi Mohamed Ben Abdellah, B. P. 1796 Fes Atlas, Morocco}
\email{\textcolor[rgb]{0.00,0.00,0.84}{rossafimohamed@gmail.com; mohamed.rossafi@usmba.ac.ma}}



\date{Received: 
\newline \indent $^{*}$Corresponding author}

\begin{abstract}
Frame theory has a great revolution in recent years. This new Theory have been extended from Hilbert spaces to Hilbert  $C^{\ast}$-modules. In this paper, we introduce the notion of dual $\ast$-$K$-g-frames in Hilbert $\mathcal{A}$-modules. Lastly we study $\ast$-$K$-g-frames in tensor product of Hilbert $C^{\ast}$-Modules and we establish some new results.
\end{abstract} \maketitle

\section{Introduction and preliminaries}
In 1952, frames in Hilbert spaces were introduced by Duffin and Schaeffer \cite{Duf} in the study of nonharmonic Fourier series. Frames possess many nice properties which make them very useful in wavelet analysis, irregular sampling theory, signal processing and many other fields.\\ 
In 2000, Frank and Larson \cite{F4} have extended the theory for the elements of $C^{\ast}$-algebras and Hilbert $C^{\ast}$-modules. Eventually, frames with $C^{\ast}$-valued bounds in Hilbert $C^{\ast}$-modules have been considered in \cite{F2}.\\
The theory of frames has been generalized rapidly and there are various generalizations of frames in Hilbert spaces and Hilbert $C^{\ast}$-modules.

In this paper, we introduce the notion of dual $\ast$-$K$-g-frames in Hilbert $C^{\ast}$-modules. Lastly we study $\ast$-$K$-g-frames in tensor product of Hilbert $\ast$-Hilbert Modules and we establish some results.

Let $I$ and $J$ be countable index sets. In this section we briefly recall the definitions and basic properties of $C^{\ast}$-algebra and Hilbert $C^{\ast}$-modules. For information about frames in Hilbert spaces we refer to \cite{Ch}. Our references for $C^{\ast}$-algebras are \cite{Dav,Con}. For a $C^{\ast}$-algebra $\mathcal{A}$, an element $a\in\mathcal{A}$ is positive ($a\geq 0$) if $a=a^{\ast}$ and $sp(a)\subset\mathbf{R^{+}}$. $\mathcal{A}^{+}$ denotes the set of positive elements of $\mathcal{A}$.
\begin{definition}\cite{BA}.
	
Let $ \mathcal{A} $ be a unital $C^{\ast}$-algebra and $\mathcal{H}$ be a left $ \mathcal{A} $-module, such that the linear structures of $\mathcal{A}$ and $ \mathcal{H} $ are compatible. $\mathcal{H}$ is a pre-Hilbert $\mathcal{A}$-module if $\mathcal{H}$ is equipped with an $\mathcal{A}$-valued inner product $\langle.,.\rangle_{\mathcal{A}} :\mathcal{H}\times\mathcal{H}\rightarrow\mathcal{A}$, such that is sesquilinear, positive definite and respects the module action. In the other words,
	\begin{itemize}
		\item [(i)] $ \langle x,x\rangle_{\mathcal{A}}\geq0 $ for all $ x\in\mathcal{H} $ and $ \langle x,x\rangle_{\mathcal{A}}=0$ if and only if $x=0$.
		\item [(ii)] $\langle ax+y,z\rangle_{\mathcal{A}}=a\langle x,z\rangle_{\mathcal{A}}+\langle y,z\rangle_{\mathcal{A}}$ for all $a\in\mathcal{A}$ and $x,y,z\in\mathcal{H}$.
		\item[(iii)] $ \langle x,y\rangle_{\mathcal{A}}=\langle y,x\rangle_{\mathcal{A}}^{\ast} $ for all $x,y\in\mathcal{H}$.
	\end{itemize}	 
	For $x\in\mathcal{H}, $ we define $\|x\|=\|\langle x,x\rangle_{\mathcal{A}}\|^{\frac{1}{2}}$. If $\mathcal{H}$ is complete with $\|.\|$, it is called a Hilbert $\mathcal{A}$-module or a Hilbert $C^{\ast}$-module over $\mathcal{A}$. For every $a$ in $C^{\ast}$-algebra $\mathcal{A}$, we have $|a|=(a^{\ast}a)^{\frac{1}{2}}$ and the $\mathcal{A}$-valued norm on $\mathcal{H}$ is defined by $\|x\|=\langle x, x\rangle_{\mathcal{A}}^{\frac{1}{2}}$ for $x\in\mathcal{H}$.
	
	Let $\mathcal{H}$ and $\mathcal{K}$ be two Hilbert $\mathcal{A}$-modules, a map $T:\mathcal{H}\rightarrow\mathcal{K}$ is said to be adjointable if there exists a map $T^{\ast}:\mathcal{K}\rightarrow\mathcal{H}$ such that $\langle Tx,y\rangle_{\mathcal{A}}=\langle x,T^{\ast}y\rangle_{\mathcal{A}}$ for all $x\in\mathcal{H}$ and $y\in\mathcal{K}$.
	
	From now, let $\mathcal{H}_{i}$ be closed submodules of $\mathcal{K}$ for each $i\in I$. Also, we reserve the notation $End_{\mathcal{A}}^{\ast}(\mathcal{H},\mathcal{H}_{i})$ for the set of all adjointable operators from $\mathcal{H}$ to $\mathcal{H}_{i}$ and $End_{\mathcal{A}}^{\ast}(\mathcal{H},\mathcal{H})$ is abbreviated to $End_{\mathcal{A}}^{\ast}(\mathcal{H})$.
\end{definition}
\begin{definition}\cite{F4}
Let $ \mathcal{A} $ be a unital $C^{\ast}$-algebra and $I$ be a finite or countable index of $\mathbb{N}$. A sequence $\{x_{i}\}_{i\in I}$ of elements in a Hilbert $\mathcal{A}$-module $\mathcal{H}$ is said to be a frame if there are real constants $A,B>0$ such that 
	\begin{equation}\label{1}
		A\langle x,x\rangle_{\mathcal{A}}\leq\sum_{i\in I}\langle x,x_{i}\rangle_{\mathcal{A}}\langle x_{i},x\rangle_{\mathcal{A}}\leq B\langle x,x\rangle_{\mathcal{A}},\qquad x\in \mathcal{H}.
	\end{equation}
The numbers $A$ and $B$ are called frames bounds. The frame $\{x_{i}\}_{i\in I}$ is said to be a tight frame if $A=B$, and said to be normalized if $A=B=1$.\\
If the sum in the middle of \eqref{1} converges in norm, $\{x_{i}\}_{i\in I}$ is called standard (normalized tight).
\end{definition}

\begin{definition}\cite{A}
A $\ast$-g-frame for $\mathcal{H}$ is a collection of ordered pairs $(\Lambda_{i},\mathcal{H}_{i})_{i\in I }$ such that  
\begin{equation}\label{4}
		A\langle x,x\rangle_{\mathcal{A}} A^{\ast}\leq\sum_{i\in I}\langle \Lambda_{i}x,\Lambda_{i}x\rangle_{\mathcal{A}}\leq B\langle x,x\rangle_{\mathcal{A}} B^{\ast}, 
\end{equation}
for $x\in \mathcal{H}$ and $A$,$B$ are strictly nonzero elements of $\mathcal{A}$.\\
	The elements $A$ and $B$ are called lower and upper $\ast$-g-frame bounds, respectively. If $A=B$, the $\ast$-g-frame is called tight and it is normalized when $A=B$.
\end{definition}

\begin{definition} \cite{RK2}
Let $K\in End_{\mathcal{A}}^{\ast}(\mathcal{H})$, a sequence $\{\Lambda_{i}\in End_{A}^{\ast}(\mathcal{H},\mathcal{H}_{i}):i\in I \}$ is called a $\ast$-K-g-frame in Hilbert $\mathcal{A}$-module $\mathcal{H}$ with respect to $\{\mathcal{H}_{i}:i\in I \}$ if there exist strictly nonzero elements $A$, $B$ in $\mathcal{A}$ such that 
	\begin{equation}\label{123}
	A\langle K^{\ast}x,K^{\ast}x\rangle_{\mathcal{A}} A^{\ast}\leq\sum_{i\in I}\langle \Lambda_{i}x,\Lambda_{i}x\rangle_{\mathcal{A}}\leq B\langle x,x\rangle_{\mathcal{A}} B^{\ast}, \qquad x\in\mathcal{H}.
	\end{equation}
\end{definition}

In this case, let $T$ be an operator defined by:
\begin{align*}
T:\mathcal{H}&\longrightarrow \oplus_{i\in I}\mathcal{H}_{i}\\
x&\longrightarrow \{\Lambda_{i}x\}_{i\in I}.
\end{align*}
$T$ is called the analysis operator. By \cite{RK2}, it's a linear and bounded operator.\\
Then it's adjoint operator is $T^{\ast}$, is given by,
\begin{align*}
T^{\ast}:\oplus_{i\in I}\mathcal{H}_{i}&\longrightarrow \mathcal{H} \\
\{x_{i}\}_{i\in I}&\longrightarrow  \sum_{i\in I}\Lambda^{\ast}_{i}x_{i}.
\end{align*}
The operator $T^{\ast}$ is called the synthesis operator.\\
By composing $T$ and $T^{\ast}$, we obtain the frame operator $S$ wich's given by,
\begin{align*}
S:\mathcal{H}&\longrightarrow \mathcal{H}\\
x&\longrightarrow Sx=T^{\ast}Tx=\sum_{i\in I}\Lambda^{\ast}_{i}\Lambda_{i}x.
\end{align*}

Let $\{x_{i}\}_{i\in I}$ and $\{y_{i}\}_{i\in I}$ in $\oplus_{i\in I}\mathcal{H}_{i}$, there inner product is defined as follow,
\begin{equation*}
\langle \{x_{i}\}_{i\in I},\{y_{i}\}_{i\in I}\rangle = \sum_{i\in I}\langle x_{i},y_{i}\rangle_{\mathcal{A}}
\end{equation*}
For the Following theorem, $R(L)$ denote the range of the operator $L$ and $\overline{R(L)}$ denote it's adherence.
\begin{theorem}\label{t1}\cite{11}
Let $E, F$ and $G$ be Hilbert $\mathcal{A}$-modules. Let $K \in End^{\ast}_{\mathcal{A}}(G,F)$ and $L \in End^{\ast}_{\mathcal{A}}(E,F)$ with $\overline{R(L^{\ast})}$ is orthogonally complemented. The following statements are equivalent,
\begin{itemize}
\item [(i)] $KK^{\ast}\leq \lambda LL^{\ast}$ for some $\lambda > 0$.
\item [(ii)] There exists $\mu >0$ such that $\|K^{\ast}z\|\leq \mu \|L^{\ast}z\|$ for all $z\in F$.
\item [(iii)] There exists $D\in End^{\ast}_{\mathcal{A}}(G,E)$ such that $K=LD$.
\item [(iv)] $R(K) \subseteq R(L)$.
\end{itemize}
\end{theorem}
\begin{theorem}\label{t2}\cite{29}
Given a Hilbert $C^{\ast}$-module $\mathcal{H}$ over a $C^{\ast}$-algebra $\mathcal{A}$. If $T,S$ are in $End^{\ast}_{\mathcal{A}}(\mathcal{H})$, and $R(S)$
is closed, then the following statements are equivalent:
\begin{itemize}
	\item [(i)] $R(T) \subset R(S)$.
	\item [(ii)] $TT^{\ast} \leq  \lambda^{2}SS^{\ast}$, for some $\lambda \geq0$.
	\item [(iii)] There exists $Q \in End^{\ast}_{\mathcal{A}}(\mathcal{H})$ such that $T = SQ$.
\end{itemize}
\end{theorem}
\section{Some properties of $\ast$-K-g-frames in Hilbert $\mathcal{A}$-modules}

\begin{proposition}	
Let $\{\Lambda_{i}\in End_{\mathcal{A}}^{\ast}(\mathcal{H},\mathcal{H}_{i}),i\in I \}$ be a $\ast$-K-g-frame with bounds $A,B$ and analysis operator $T$, then :
	\begin{equation}\label{P1}
	\|AK^{\ast}f\|^{2} \leq \|\sum_{i\in I}\langle \Lambda_{i}f,\Lambda_{i}f\rangle_{\mathcal{A}}\|\leq \|Bf\|^{2} , \qquad f \in \mathcal{H}.
	\end{equation}
Conversely, if \eqref{P1} holds, for some $A,B \in \mathcal{Z}(\mathcal{A})$, and $\overline{R(T)}$ is orthogonally complemented, then $\{\Lambda_{i}\}_{i\in I}$ is a $\ast$-K-g frames ($\mathcal{Z}(\mathcal{A})$ denote the center of the $C^{\ast}$-algebra $\mathcal{A}$).
	\end{proposition}
	\begin{proof}
		
$\Rightarrow )$ 
		
Let $\{\Lambda_{i}\in End_{\mathcal{A}}^{\ast}(\mathcal{H},\mathcal{H}_{i}),i\in I \}$ be a $\ast$-K-g-frame in Hilbert $\mathcal{A}$-module $\mathcal{H}$ with bounds A and B then, for all $f$ in $\mathcal{H}$ we have,
		\begin{equation*}
		A\langle K^{\ast}f,K^{\ast}f\rangle_{\mathcal{A}} A^{\ast}\leq \sum_{i\in I}\langle \Lambda_{i}f,\Lambda_{i}f\rangle_{\mathcal{A}} \leq B\langle f,f\rangle_{\mathcal{A}} B^{\ast}.	
		\end{equation*}
Then,
		\begin{equation*}
			0\leq \langle AK^{\ast}f,AK^{\ast}f\rangle_{\mathcal{A}} \leq \sum_{i\in I}\langle \Lambda_{i}f,\Lambda_{i}f\rangle_{\mathcal{A}} \leq \langle Bf,Bf\rangle_{\mathcal{A}}.
		\end{equation*}
So,
			\begin{equation*}
		\| \langle AK^{\ast}f,AK^{\ast}f\rangle_{\mathcal{A}} \| \leq \|\sum_{i\in I}\langle \Lambda_{i}f,\Lambda_{i}f\rangle_{\mathcal{A}} \| \leq \|\langle Bf,Bf\rangle_{\mathcal{A}} \|. 
		\end{equation*}
	Finally we obtain,
			\begin{equation*}
			\|AK^{\ast}f\|^{2} \leq \|\sum_{i\in I}\langle \Lambda_{i}f,\Lambda_{i}f\rangle_{\mathcal{A}}\|\leq \|Bf\|^{2}. 
		\end{equation*}
\\
		$\Leftarrow )$ For $ f$ in $\mathcal{H}$, we have:
		\begin{equation*}
		\|Tf\|^{2} = \|\langle Tf,Tf\rangle_{\mathcal{A}} \| = \|\sum_{i\in I}\langle \Lambda_{i}f,\Lambda_{i}f\rangle_{\mathcal{A}}\| \leq \|Bf\|^{2}= \|B\|^{2}\|f\|^{2}.
		\end{equation*}
		Then,
		\begin{equation*}
		 \|T f\| \leq \|B\| \|f\|
		 \end{equation*}
		It is clearly wich show that $T$ is a bounded. On the other hand $T$ is adjointable and it's adjointable operator is $T^{\ast}$ is given by,
		\begin{equation*} 
		T^{\ast}(\{f_{i}\}_{i \in I}) = \sum_{i\in I}\Lambda_{i}^{\ast}f_{i}
		\end{equation*}
		
		By Theorem \ref{t1}, there existe $\lambda , \mu > 0 $ such that for every $ f \in \mathcal{H} $, we have,
		\begin{equation*}
		(\sqrt{\lambda}A)\langle K^{\ast}f, K^{\ast}f\rangle_{\mathcal{A}} (\sqrt{\lambda}A)^{\ast} \leq \langle\sum_{i\in I}\Lambda_{i}f,\Lambda_{i}f\rangle_{\mathcal{A}}\leq (\sqrt{\mu}B)\langle f, f\rangle_{\mathcal{A}} (\sqrt{\mu}B)^{\ast}. 
		\end{equation*}		
	\end{proof}
\begin{proposition}\label{PP2}
	
Let $\{\Lambda_{i}\in End_{\mathcal{A}}^{\ast}(\mathcal{H},\mathcal{H}_{i}), i\in I \}$  be a $\ast$-K-g-frame with bounds A and B then:
	\begin{equation}
	\|AK^{\ast}f\|^{2} \leq \|\sum_{i\in I}\langle \Lambda_{i}f,\Lambda_{i}f\rangle_{\mathcal{A}}\|\leq \|Bf\|^{2} , \qquad f \in \mathcal{H}.
	\end{equation}
	Conversely, if (3.2) holds, for some $A,B \in \mathcal{Z}(\mathcal{A})$, and $T$ has closed range, then $\{\Lambda_{i}\}_{i\in I}$ is a $\ast$-K-g frames.
	
\end{proposition}
\begin{proof}	
Similar to the proof of the last proposition.
\end{proof}
\begin{proposition}\label{PP3}
	
	Let $K,L \in End_{\mathcal{A}}^{\ast}(\mathcal{H})$ and $\{\Lambda_{i}\}_{i\in I}$ be a $\ast$-K-g-frames with bounds $A$ and $B$, then,
	\begin{itemize}
		\item [1)] If $U : \mathcal{H}\longmapsto \mathcal{H}$ is a co-isometry, such that $KU=UK$ then $\{\Lambda_{i}U^{\ast}\}_{i\in I}$ is a $\ast$-K-g-frames. 
		\item [2)] $\{\Lambda_{i}L^{\ast}\}_{i\in I}$ is a $\ast$-LK-g-frames with bounds $A$ and $B\|L\|$ respectively.
		\item [3)] $\{\Lambda_{i}(L^{\ast})^{n}\}_{i\in I}$ is a $\ast$-$L^{n}K$-$g$-frames, for any $n \in \mathbb{N}$.
		\item [4)]If $R(L) \subseteq R(K)$ and $K$ has closed range then $\{\Lambda_{i}\}_{i\in I}$ is also\\ $\ast$-L-g-frames.		
	\end{itemize}
\begin{proof}
	\begin{itemize}
		\item [1)] Let $\{\Lambda_{i}\}_{i\in I}$ be a $\ast$-K-g-frames with bounds $A$ and $B$, then,
		\begin{equation*}
		A\langle K^{\ast}f,K^{\ast}f\rangle A^{\ast} \leq \sum_{i\in I} \langle \Lambda_{i}f,\Lambda_{i}f\rangle \leq B\langle  f,f\rangle B^{\ast}, \qquad f \in \mathcal{H}.
		\end{equation*}
		
		On one hand, we have,
		\begin{equation*}
		\sum_{i\in I} \langle \Lambda_{i}U^{\ast}f,\Lambda_{i}U^{\ast}f\rangle \leq B\langle  U^{\ast}f,U^{\ast}f\rangle B^{\ast}=B\langle  f,f\rangle B^{\ast}, \qquad f \in \mathcal{H}.
		\end{equation*}
		
		On the other hand, we have , 
		
		\begin{align*}
		A\langle K^{\ast}f,K^{\ast}f\rangle A^{\ast}&=A\langle U^{\ast}K^{\ast}f,U^{\ast}K^{\ast}f\rangle A^{\ast}\\
		&=A\langle K^{\ast}U^{\ast}f,K^{\ast}U^{\ast}f\rangle A^{\ast}\\
		&\leq \sum_{i\in I} \langle \Lambda_{i}U^{\ast}f,\Lambda_{i}U^{\ast}f\rangle \qquad f\in \mathcal {H}.
		\end{align*}

		Wich proves (1).

		\item [2)] For any $f$ in $\mathcal{H}$, we have, 
		
		\begin{equation*}
		A\langle (LK)^{\ast}f,(LK)^{\ast}f\rangle A^{\ast} = A\langle K^{\ast}L^{\ast}f,K^{\ast}L^{\ast}f\rangle A^{\ast} \leq  \sum_{i\in I} \langle \Lambda_{i}L^{\ast}f,\Lambda_{i}L^{\ast}f\rangle \leq B\langle L^{\ast}f,L^{\ast}f\rangle B^{\ast}.
		\end{equation*}
		
		Since,
		
		\begin{equation*}
		B\langle L^{\ast}f,L^{\ast}f\rangle B^{\ast}\leq (B\|L\|)\langle f, f \rangle (B\|L\|)^{\ast},
		\end{equation*}
		
		then,
		
		\begin{equation*}
		\sum_{i\in I} \langle \Lambda_{i}L^{\ast}f,\Lambda_{i}L^{\ast}f\rangle\leq (B\|L\|)\langle f, f \rangle (B\|L\|)^{\ast}.
		\end{equation*}
				
		\item [3)] Obvious by (2).
		\item [4)]	By Theorem \ref{t2}, there exists a positive real number $\lambda > 0 $ such that for all $f \in \mathcal{H}$, we have,
		\begin{equation*}
		\lambda LL^{\ast}f \leq KK^{\ast}f.
		\end{equation*}
	 Thus,
		\begin{equation*}
		(\sqrt{\lambda}A)\langle L^{\ast}f,L^{\ast}f\rangle (\sqrt{\lambda}A)^{\ast} \leq A\langle K^{\ast}f,K^{\ast}f\rangle A^{\ast}\leq \sum_{i\in I} \langle \Lambda_{i}f,\Lambda_{i}f\rangle \leq B\langle  f,f\rangle B^{\ast}.
		\end{equation*}
		
	\end{itemize}
\end{proof}

\end{proposition}
\section{Perturbation of $\ast$-K-g-frames}

\begin{theorem}	
Assume that $K, L \in End_{\mathcal{A}}^{\ast}(\mathcal{H})$, with $R(L) \subseteq R(K)$ and K has a closed range.
Let $\{\Lambda_{i}\}_{i\in I}$ be a $\ast$-K-g frame with bounds A and B.\\

If there exists a constant $M > 0$ such that for all $f$ in $\mathcal{H}$ :
	\begin{equation*}
\|\sum_{i\in I}\langle (\Lambda_{i}-\Gamma_{i})f,(\Lambda_{i}-\Gamma_{i})f\rangle_{\mathcal{A}}\| \leq M\min{\{\|\sum_{i\in I}\langle \Lambda_{i}f,\Lambda_{i}f\rangle_{\mathcal{A}}\|;\|\sum_{i\in I}\langle \Gamma_{i}f,\Gamma_{i}f\rangle_{\mathcal{A}}\|\}}.
	\end{equation*}
	Then $\{\Gamma_{i}\}_{i\in I}$ is a $\ast$-L-g-frames.
	If K is a co-isometry, $R(K) \subseteq R(L)$ and R(L) is closed then the converse is valid.
\end{theorem}

\begin{proof} 
Let $f\in \mathcal{H}$, we have :
	\begin{align*}
	\|(\Gamma_{i}f)_{i\in I}\|	
	&=\|(\Gamma_{i}f)_{i\in I} - (\Lambda_{i}f)_{i\in I} + (\Lambda_{i}f)_{i\in I}\|\\
	&\leq \|(\Gamma_{i}f)_{i\in I} - (\Lambda_{i}f)_{i\in I}\| + \| (\Lambda_{i}f)_{i\in I}\|. 
		\end{align*}
On one hand, we have,
		\begin{align*}
	\|\sum_{i\in I}\langle \Gamma_{i}f,\Gamma_{i}f\rangle_{\mathcal{A}}\|^{1/2}&\leq \sqrt{M}\|\sum_{i\in I}\langle \Lambda_{i}f,\Lambda_{i}f\rangle_{\mathcal{A}}\|^{1/2} + \|\sum_{i\in I}\langle \Lambda_{i}f,\Lambda_{i}f\rangle_{\mathcal{A}}\|^{1/2}\\
		&\leq (1+\sqrt{M})(\|\sum_{i\in I}\langle \Lambda_{i}f,\Lambda_{i}f\rangle_{\mathcal{A}}\|^{1/2})\\ 
&\leq (1+\sqrt{M})(\|B\|\|f\|).
		\end{align*}
		
	On other hand, we have:
	
	\begin{align*}
	\| \sum_{i\in I}\langle \Lambda_{i}f, \Lambda_{i}f \rangle\|^{1/2} &\leq \|\sum_{i\in I}\langle (\Lambda_{i}-\Gamma_{i})f,(\Lambda_{i}-\Gamma_{i})f\rangle_{\mathcal{A}}\|^{1/2} + \|\sum_{i\in I}\langle \Gamma_{i}f,\Gamma_{i}f\rangle_{\mathcal{A}}\|^{1/2}\\
	&\leq \sqrt{M}\|\sum_{i\in I}\langle \Gamma_{i}f,\Gamma_{i}f\rangle_{\mathcal{A}}\|^{1/2} + \|\sum_{i\in I}\langle \Gamma_{i}f,\Gamma_{i}f\rangle_{\mathcal{A}}\|^{1/2} \\
	&\leq (1+\sqrt{M})(\|\sum_{i\in I}\langle \Gamma_{i}f,\Gamma_{i}f\rangle_{\mathcal{A}}\|^{1/2}).
	\end{align*}
	
	Thus, 
	
	\begin{equation*}
	\frac{\|A\|^{2}}{(1+\sqrt{M})^{2}}\|K^{\ast}f\|^{2}\leq \frac{1}{(1+\sqrt{M})^{2}}\|\sum_{i\in I}\langle \Lambda_{i}f,\Lambda_{i}f\rangle_{\mathcal{A}}\| \leq \|\sum_{i\in I}\langle \Gamma_{i}f,\Gamma_{i}f\rangle_{\mathcal{A}}\|.
	\end{equation*}
	Finally, we obtain
	\begin{equation*}
	\frac{\|A\|^{2}}{(1+\sqrt{M})^{2}}\|K^{\ast}f\|^{2}\leq \|\sum_{i\in I}\langle \Gamma_{i}f,\Gamma_{i}f\rangle_{\mathcal{A}}\|\leq (1+\sqrt{M})^{2}(\|B\|^{2}\|f\|^{2}).
	\end{equation*}
	
	Then by Proposion \ref{PP2} we conclude that  $\{\Gamma_{i}\}_{i \in I}$ is a $\ast$-K-g-frame. On the other hand, since $R(L) \subseteq R(K)$ and $K$ has a closed range, by proposition \ref{PP3}, we conclude that $\{\Gamma_{i}\}_{i \in I}$ is a $\ast$-L-g-frames.
	
	Conversely;
	since that $\{\Gamma_{i}\}_{i \in I}$  is a $\ast$-L-g-frames with bounds C and D.\\
	On one hand, we have for all $f$ in $\mathcal{H}$,
	
		\begin{align*}
	\|\sum_{i\in I}\langle (\Lambda_{i}-\Gamma_{i})f,(\Lambda_{i}-\Gamma_{i})f\rangle_{\mathcal{A}}\|^{1/2} &\leq \|\sum_{i\in I}\langle (\Lambda_{i}f,(\Lambda_{i})f\rangle_{\mathcal{A}}\|^{1/2} + \|\sum_{i\in I}\langle (\Gamma_{i}f,(\Gamma_{i})f\rangle_{\mathcal{A}}\|^{1/2}\\
	&\leq \|\sum_{i\in I}\langle (\Lambda_{i}f,(\Lambda_{i})f\rangle_{\mathcal{A}}\|^{1/2} + \|D\|\|f\|\\
	& \leq \|\sum_{i\in I}\langle (\Lambda_{i}f,(\Lambda_{i})f\rangle_{\mathcal{A}}\|^{1/2} + \|D\|\|K^{\ast}f\|\\
	&\leq \|\sum_{i\in I}\langle (\Lambda_{i}f,(\Lambda_{i})f\rangle_{\mathcal{A}}\|^{1/2} + \frac{\|D\|}{\|A\|}\|\sum_{i\in I}\langle (\Lambda_{i}f,(\Lambda_{i})f\rangle_{\mathcal{A}}\|^{1/2}\\
	&\leq (1+\frac{\|D\|}{\|A\|})\|\sum_{i\in I}\langle (\Lambda_{i}f,(\Lambda_{i})f\rangle_{\mathcal{A}}\|^{1/2}.
	\end{align*}	
	On the other hand we have,
	
	\begin{align*}
	\|\sum_{i\in I}\langle (\Lambda_{i}-\Gamma_{i})f,(\Lambda_{i}-\Gamma_{i})f\rangle_{\mathcal{A}}\|^{1/2} &\leq \|\sum_{i\in I}\langle \Lambda_{i}f,\Lambda_{i}f\rangle_{\mathcal{A}}\|^{1/2} + \|\sum_{i\in I}\langle \Gamma_{i}f,\Gamma_{i}f\rangle_{\mathcal{A}}\|^{1/2}\\
&\leq \|\sum_{i\in I}\langle \Gamma_{i}f,\Gamma_{i}f\rangle_{\mathcal{A}}\|^{1/2} + \|B\|\|f\|\\
&\leq \|\sum_{i\in I}\langle \Gamma_{i}f,\Gamma_{i}f\rangle_{\mathcal{A}}\|^{1/2} + \|B\|\|K^{\ast}f\|.
	\end{align*}
By Theorem \ref{t2}, there exist $\lambda \geq 0$ such that 
\begin{equation*}
KK^{\ast}\leq \lambda^{2}LL^{\ast},
\end{equation*}
 then
	\begin{equation*}
	 \|K^{\ast}f\|\leq \lambda \|L^{\ast}f\|.
	\end{equation*}
We conclude that,
	
	\begin{align*}
	\|\sum_{i\in I}\langle (\Lambda_{i}-\Gamma_{i})f,(\Lambda_{i}-\Gamma_{i})f\rangle_{\mathcal{A}}\|^{1/2}&\leq \|\sum_{i\in I}\langle \Gamma_{i}f,\Gamma_{i}f\rangle_{\mathcal{A}}\|^{1/2} + \lambda\|B\|\|L^{\ast}f\|\\
	&\leq \frac{1+\lambda\|B\|}{\|C\|}\|\sum_{i\in I}\langle \Gamma_{i}f,\Gamma_{i}f\rangle_{\mathcal{A}}\|^{1/2}.
	\end{align*}
	
	Then 
	
	\begin{equation*}
	\|\sum_{i\in I}\langle (\Lambda_{i}-\Gamma_{i})f,(\Lambda_{i}-\Gamma_{i})f\rangle_{\mathcal{A}}\| \leq M\min{\{\|\sum_{i\in I}\langle \Lambda_{i}f,\Lambda_{i}f\rangle_{\mathcal{A}}\|;\|\sum_{i\in I}\langle \Gamma_{i}f,\Gamma_{i}f\rangle_{\mathcal{A}}\|\}}
	\end{equation*}
	
	such that $M=min(\frac{1+\lambda\|B\|}{\|C\|},(1+\frac{\|D\|}{\|A\|}))$.
	
	\end{proof}

\section{The dual of $\ast$-K-g-frames}

\begin{definition}	
Let $\mathcal{A}$ be a unital $C^{\ast}$-algebra and let  $\mathcal{H}$ an Hilbert $\mathcal{A}$-module over a unital $C^{\ast}$-algebra and $K \in End_{\mathcal{A}}^{\ast}(\mathcal{H})$. Let $\{ \Lambda_{i} \in End_{\mathcal{A}}^{\ast}(\mathcal{H},\mathcal{H}_{i}), i\in I \}$ be a $\ast$-K-g-frames in  $\mathcal{H}$ associated to $\{\mathcal{H}_{i}\}_{i\in I}$. A $\ast$-K-g-Bessel sequence $\{ \Gamma_{i} \in End_{\mathcal{A}}^{\ast}(\mathcal{H},\mathcal{H}_{i}), i\in I \}$ is called a dual $\ast$-K-g-frames for $\{ \Lambda_{i}\}_{i \in I}$ if,
\begin{equation*}
Kf=\sum_{i\in I}\Lambda_{i}^{\ast}\Gamma_{i}f, \qquad f\in \mathcal{H}. 	
\end{equation*}

\end{definition}

\begin{example}	
Let $K {\in End_\mathcal{A}}^{\ast}(\mathcal{H})$ be a surjective operator and $\{\Lambda_{i} \in End_{\mathcal{A}}^{\ast}(\mathcal{H},\mathcal{H}_{i})\}_{i\in I}$ be a $\ast$-K-g-frames in $\mathcal{H}$ associated to $\{\mathcal{H}_{i}\}_{i\in I}$ with $\ast$-K-g-frames operator $S$.\\
By  \cite{RK2}, $S$ is invertible.
	
	For all $f \in \mathcal{H}$ we have :
	\begin{equation*}
	Sf=\sum_{i\in I}\Lambda_{i}^{\ast}\Lambda_{i}f.
	\end{equation*}
	Then,
	\begin{equation*}
	Kf=\sum_{i\in I}\Lambda_{i}^{\ast}\Lambda_{i}S^{-1}Kf.
	\end{equation*}
	
Then the sequence  $\{ \Lambda_{i}S^{-1}K \in End_{\mathcal{A}}^{\ast}(\mathcal{H},\mathcal{H}_{i}), i\in I \}$ is a ${\ast}$-K-g frame, and is a dual ${\ast}$-K-g frame of  $\{ \Lambda_{i} \in End_{\mathcal{A}}^{\ast}(\mathcal{H},\mathcal{H}_{i}), i\in I \}$ 
\end{example}
\begin{example}
Let $\mathcal{H}$ be a Hilbert $C^{\ast}$-modules and let $(e_{i})_{i\geq 1}$ be an orthonormal basis for $\mathcal{H}$.

We define an operator $K$ by,
\begin{align*}
K : \mathcal{H}&\longrightarrow \mathcal{H}\\
e_{i}&\longrightarrow Ke_{i}, \quad Ke_{2i}=e_{2i}+e_{2i-1}
\end{align*}
By a simple calculation, we have,
\begin{equation*}
Kf=\sum_{i\geq 1}\langle f,e_{2i}\rangle (e_{2i}+e_{2i-1}),
\end{equation*}
and,
\begin{equation*}
K^{\ast}f=\sum_{i\geq 1}\langle f,e_{2i}+e_{2i-1}\rangle e_{2i}.
\end{equation*}
Let $\mathcal{H}_{i}=\overline{span}(e_{i} + e_{i+1})$ \quad for any $i\geq 1$\\
We define the sequences of operators $\{\Lambda_{i}\}_{i\geq 1}$ and $\{\Gamma_{i}\}_{i\geq 1}$ by,
\begin{align*}
\Lambda_{i} : \mathcal{H}_{i}&\longrightarrow \mathcal{H}\\
f&\longrightarrow \Lambda_{i} f=\langle f,e_{2i}+e_{2i-1}\rangle e_{i}
\end{align*}
and,
\begin{align*}
\Gamma_{i} : \mathcal{H}_{i}&\longrightarrow \mathcal{H}\\
f&\longrightarrow \Gamma_{i} f=\langle f,e_{2i}\rangle e_{i}.
\end{align*}
It easy to schow that,
\begin{equation*}
\langle K^{\ast}f,K^{\ast}f\rangle = \sum_{i\geq 1}\langle \Lambda_{i}f,\Lambda_{i}f\rangle 
\end{equation*}
Then $\{\Lambda_{i}\}_{i\geq 1}$ is a normalised tight $\ast$-k-g frames for $\mathcal{H}$. Moreover, we have
\begin{equation*}
Kf=\sum_{i\geq 1}\Lambda^{\ast}_{i}\Gamma_{i}f.
\end{equation*}
Which shows that $\{\Gamma_{i}\}_{i\geq 1}$ is the dual of $\{\Lambda_{i}\}_{i\geq 1}$.
\end{example} 

	
\section{Tensor Product}
	
\begin{theorem}		
Let $\{\Lambda_{i}\}_{i\in I} $ and $\{\Gamma_{j}\}_{j\in J} $ be $\ast$-K-g-frames and $\ast$-L-g-frames respectively in $\mathcal{H}$, with the duals $\{\tilde{\Lambda_{i}}\}_{i\in I}$ and  $\{\tilde{\Gamma_{j}}\}_{j\in J} $ respectively. Then $\{\tilde{\Lambda_{i}}\otimes\tilde{\Gamma_{j}}\}_{i\in I,j\in J}$ is a dual of $\{\Lambda_{i}\otimes\Gamma_{j}\}_{i\in I,j\in J}$.
\end{theorem}
\begin{proof}		
By definition, for all $x \in \mathcal{H}$ and $ y \in \mathcal{K}$ we have,
\begin{equation*}
\sum_{i\in I}\Lambda_{i}^{\ast}\tilde{\Lambda_{i}}x=Kx  \qquad and \qquad \sum_{j\in J}\Gamma_{j}^{\ast}\tilde{\Gamma_{j}}y=Ly.  
\end{equation*}

Then : 
\begin{align*}
(K\otimes L)(x\otimes y)&= Kx\otimes Ly \\
&= \sum_{i\in I}\Lambda_{i}^{\ast}\tilde{\Lambda_{i}}x\otimes\sum_{j\in J}\Gamma_{j}^{\ast}\tilde{\Gamma_{j}}y\\
&=\sum_{i,j\in I,J}\Lambda_{i}^{\ast}\tilde{\Lambda_{i}}x\otimes\Gamma_{j}^{\ast}\tilde{\Gamma_{j}}y\\
&=\sum_{i,j\in I,J}(\Lambda_{i}^{\ast}\otimes\Gamma_{j}^{\ast})(\tilde{\Lambda_{i}}x\otimes\tilde{\Gamma_{j}}y)\\
&=\sum_{i,j\in I,J}(\Lambda_{i}\otimes\Gamma_{j})^{\ast}(\tilde{\Lambda_{i}}\otimes\tilde{\Gamma_{j}})(x\otimes y).
\end{align*}
		Then $\{\tilde{\Lambda_{i}}\otimes\tilde{\Gamma_{j}}\}_{i,j \in I,J}$ is a dual of $\{\Lambda_{i}\otimes\Gamma_{j}\}_{i,j \in I,J}$.
	\end{proof}
	\begin{corollary}		
		Let $\{\Lambda_{i,j}\}_{0\leq i\leq n; j\in J}$ be a family of $\ast$-$K_{i}$-g-frames, such $ 0 \leq i \leq n $ and $\{\tilde{\Lambda}_{i,j}\}_{0\leq i\leq n; j\in J}$ their dual, then $\{\tilde{\Lambda}_{0,j}\otimes \tilde{\Lambda}_{1,j}\otimes......\otimes\tilde{\Lambda}_{n,j}\}_{j\in J}$ is a dual of $\{\Lambda_{0,j}\otimes \Lambda_{1,j}\otimes......\otimes\Lambda_{n,j}\}_{j\in J}$.
	\end{corollary}
\begin{proof}
	Obvious by last theorem.
\end{proof}

\bibliographystyle{amsplain}

\end{document}